\documentclass{amsart}
\usepackage{amsxtra, microtype}
\usepackage{stmaryrd }
\usepackage{mathrsfs}
\usepackage[margin=1.45in]{geometry}
\usepackage{amssymb,amsmath,amsthm,stmaryrd,latexsym,wasysym}
\usepackage[all]{xy}
\usepackage{hyperref}
\usepackage{mathtools}
\usepackage{tikz}
\usepackage{tikz-cd}
\usetikzlibrary{decorations.pathmorphing}
\usetikzlibrary{arrows}

\usepackage{quiver}
\usepackage{scalerel}

\usepackage{stackrel}

\usepackage{cancel}
\usepackage{appendix}
\theoremstyle{plain}
\usepackage{xcolor}
\usepackage{xspace}

\newtheorem{theorem}{Theorem}[section]

\newtheorem{proposition}[theorem]{Proposition}
\newtheorem{corollary}[theorem]{Corollary}
 \theoremstyle{definition}

\newtheorem{thm*}[]{Theorem}
\newcommand{\lie}[1]{\mathfrak{#1}}

\newcommand{\R}{\mathbb{R}} 
 
\newcommand{\inv}{^{-1}}
\newcommand{\N}{\mathbb{N}} 

\newcommand{\mx}{\mathfrak{X}} 

\newcommand{\ldr}[1]{{{\pounds}}_{#1}}

\newcommand{\an}[1]{\arrowvert_{#1}}

\DeclareMathOperator{\pr}{pr}

\DeclareMathOperator {\id}    {id}

\allowdisplaybreaks

\begin{document}
\title{On the flows of linear vector fields}


\author{M. Jotz} \address{Institut f\"ur Mathematik,
Julius-Maximilians-Universit\"at W\"urzburg, Germany}
\email{madeleine.jotz@uni-wuerzburg.de}

\begin{abstract}
This note provides a detailed proof of the fact that a linear vector field on a vector bundle has a flow by vector bundle isomorphisms. 
It then implies easily the existence of global solutions to linear non-autonomous ODE's, with a standard time-dependent flow construction. As a further application, a simple proof of the smooth triviality of vector bundles over contractible bases is given. Finally, a detailed elementary proof of the isomorphy (as Lie algebras) of all fibers of the kernel of a transitive Lie algebroid is given.
 \end{abstract}
\maketitle

\tableofcontents

\section{Introduction}
This note shows in detail that the flow  of a linear vector field on a vector bundle is a flow by isomorphisms of the vector bundle.
This fact is standard, see e.g.~\cite[Proposition 3.4.2]{Mackenzie05}, \cite[Proposition 2.2]{MaXU98}, \cite[Proposition 6.1]{Kosmann16},
but a precise statement like the following (see Theorem \ref{lemma_linear_flow} below) and a detailed proof are not easy to find in the literature.  
\begin{thm*}\label{lemma_linear_flow_intro}
Let $q\colon E\to M$ be a smooth vector bundle and let 
$\chi\in\mx^l(E)$ be a linear vector field 
over a vector field $X\in\mx(M)$. Let $\phi\colon \Omega\to M$ be the flow of $X$ on its open flow domain $\Omega\subseteq \R\times M$ containing $\{0\}\times M$. Then the flow $\Phi$ of $\chi$
is defined on 
\[ \widetilde\Omega:=\left\{ (t, e)\in \R\times E \mid (t,q(e))\in \Omega
\right\}
\]
and for each $t\in\R$, the map
\[ \Phi_t\colon q\inv(M_t)\to q\inv(M_{-t})
\]
is an isomorphism of vector bundles over the diffeomorphism $\phi_t\colon M_t\to M_{-t}$, where 
$M_t\subseteq M$ is the open subset
\[ M_t:=\{x\in M\mid (t,x)\in\Omega\}.
\]
\end{thm*}

This paper, which is extracted from lecture notes by the author, discusses this theorem and its proof for future reference, in particular because it is needed in \cite{JoMa24} for the proof of the homotopy invariance of twisted Lie algebroid cohomology. While Theorem \ref{lemma_linear_flow_intro} can be deduced from the existence of global solutions to linear autonomous ordinary differential equations, this paper proves it only using standard results on flows of vector fields, and then deduces from it the existence of global solutions to linear autonomous ODE's, with a linear time-dependent vector field construction.

The smooth triviality of vector bundles over contractible bases -- a folklore fact, see e.g.~\cite[Lemma 2.5]{HoMe21}, \cite[Corollary 15.10]{Gloeckner22} -- follows in a straightforward manner from Theorem 
\ref{lemma_linear_flow_intro}, see Corollary \ref{cor_standard_vb_over_int} below and the discussion following it. 

The linear flows of the duals of linear derivations and of tensor products of derivations are given as well. Finally, an elementary proof 
of the isomorphy (as Lie algebras) of all fibers of the kernel of a transitive Lie algebroid is given (see e.g.~\cite[Theorem 6.5.1]{Mackenzie05}, \cite[Proposition 3.6]{Meinrenken21}), using the latter results on linear flows of tensor products.

\subsection*{Notation and conventions}

  All manifolds and vector bundles in this note are smooth and real.
  Vector bundle projections are written $q_E\colon E\to M$ if not stated otherwise, except for  $p_M\colon TM\to M$, which stands for the projection of the tangent bundle of a smooth manifold $M$.  Given a section
  $\varepsilon$ of $E^*$, the map $\ell_\varepsilon\colon E\to \R$ is
  the linear function associated to it, i.e.~the function defined by
  $e_x\mapsto \langle \varepsilon(x), e_x\rangle$ for all $e_x\in E\arrowvert_{x}$, $x\in M$.
  The set of global sections of a vector bundle $E\to M$ is denoted by
  $\Gamma(E)$ and $\mx(M)=\Gamma(TM)$ is the space of smooth vector fields on a
  smooth manifold $M$.

\subsection*{Acknowledgement}
The author thanks Rosa Marchesini for their stimulating collaboration on \cite{JoMa24}, that led her to write this short note for her students.
She  warmly thanks as well anonymous referees
for their careful reading and useful suggestions.

\section{Linear derivations versus linear vector fields on vector bundles}
This section explains in detail the tangent prolongation of a vector bundle, defines linear vector fields on vector bundles and describes their equivalence with linear derivations.
\subsection{The tangent prolongation of a smooth vector bundle}
Consider a vector bundle $q:=q_E\colon E\to M$ with addition $+\colon E\times_ME\to E$. Then the tangent bundle of $E$ is a vector bundle $TE\to E$, since $E$ is a smooth manifold. 
Consider the smooth map $Tq\colon TE\to TM$. Take an open subset $U\subseteq M$ such that $U$ is a chart domain for a chart $\varphi$ of  $M$ and a trivialising set for $E\to M$. Take a smooth frame $(e_1,\ldots, e_k)$ of $E\an U=q^{-1}(U)$ and build the dual frame $(\sigma_1,\ldots, \sigma_k)$ of $E^*\an{U}$.
Consider the smooth chart $\tilde\varphi\colon q\inv(U)\to \varphi(U)\times \R^k$ of $E$ defined by 
\[ e 
\quad \mapsto\quad 
\left(\varphi(q(e)), \ell_{\sigma_1}(e), \ldots,  \ell_{\sigma_k}(e)\right)
\]
and the induced coordinate functions $q^*\varphi_1, \ldots, q^*\varphi_n, \ell_{\sigma_1}, \ldots, \ell_{\sigma_k}$ on $E\an{U}$.
Then 
\[\begin{tikzcd}
	{q^{-1}(U)} & {\varphi(U)\times\mathbb R^k} \\
	U & {\varphi(U)}
	\arrow["{\tilde\varphi}", from=1-1, to=1-2]
	\arrow["q"', from=1-1, to=2-1]
	\arrow["{\pr_{\varphi(U)}}", from=1-2, to=2-2]
	\arrow["\varphi"', from=2-1, to=2-2]
\end{tikzcd}\]
and so 
\[\begin{tikzcd}
	{T(q^{-1}(U))=(Tq)^{-1}(TU)} && {T(\varphi(U)\times\mathbb R^k)=\varphi(U)\times\mathbb R^n\times\mathbb R^k\times\mathbb R^k} \\
	TU && {T(\varphi(U))=\varphi(U)\times\mathbb R^n}
	\arrow["{T\tilde\varphi}", from=1-1, to=1-3]
	\arrow["Tq"', from=1-1, to=2-1]
	\arrow["{\pr_{T(\varphi(U))}=T\pr_{\varphi(U)}}", from=1-3, to=2-3]
	\arrow["{T\varphi}"', from=2-1, to=2-3]
\end{tikzcd}\]
which shows that 
\[ (T\varphi\circ Tq\circ T\tilde \varphi\inv)(x,u,v,w)=\pr_{T(\varphi(U))}(x,u,v,w)=(x,u)
\]
for all $(x,u,v,w)\in T(\varphi(U))\times \R^k\times \R^k$. Hence $Tq$ is a surjective submersion, since in adapted charts, it is just a projection on the first $2n$ coordinates.
Furthermore, each fiber of $Tq$ is equipped with the structure of a vector space of dimension $2k$. 

For each $T\varphi\inv(x,u)=:u_x\in T_{\varphi\inv(x)}M$, the map
\[ T+\colon (Tq)\inv(u_x)\times (Tq)\inv(u_x)\to (Tq)\inv(u_x)
\]
sends $T(\tilde\varphi\inv)(x,u,v,w)$ and $T(\tilde\varphi\inv)(x,u,v',w')$ to
\begin{equation*}
\begin{split}
& T_{(\tilde\varphi\inv(x,v), \tilde\varphi\inv(x,v'))}+\left(T(\tilde\varphi\inv)(x,u,v,w), T(\tilde\varphi\inv)(x,u,v',w')\right)\\
&=\left.\frac{d}{dt}\right\an{t=0}\tilde\varphi\inv(x+tu, v+tw)+\tilde\varphi\inv(x+tu, v'+tw')\\
&=\left.\frac{d}{dt}\right\an{t=0}\tilde\varphi\inv(x+tu,  v+v'+t(w+w'))\\
&=T\tilde\varphi\inv(x,u,v+v',w+w').
\end{split}
\end{equation*}
This shows that the canonical addition in the fibers of $Tq$ given by the chart above coincides with $T\!+$ restricted to these fibers, and the trivialisation $\phi\colon q\inv(U)\to U\times \R^k$, $e\mapsto (q(e), \ell_{\sigma_1}(e), \ldots, \ell_{\sigma_k}(e))$, gives further a trivialisation
\[ T\phi\colon T(q\inv(U))\overset{\sim}{\longrightarrow} TU\times \R^k\times \R^k,
\]
which is an isomorphism of vector spaces in each fiber of $Tq$.
This shows that $Tq\colon TE\to TM$ is a smooth vector bundle of rank twice the rank of $E$. The space $TE$ fits so in  a square of vector bundles 
\[\begin{tikzcd}
	{TE} & {E} \\
	{TM} & {M}
	\arrow["{Tq}"', from=1-1, to=2-1]
	\arrow["p_E", from=1-1, to=1-2]
	\arrow["{q}", from=1-2, to=2-2]
	\arrow["{p_M}", from=2-1, to=2-2]
\end{tikzcd}.\]
Take $x\in M$ and $e$ and $e'\in E_x$. Take further $v_e, w_e\in T_eE$ and $u_{e'}, y_{e'}\in T_{e'}E$ such that 
\[ Tqv_e=Tqu_{e'}=:\alpha_x\in T_xM \quad \text{ and } \quad Tqw_e=Tqy_{e'}=:\beta_x\in T_xM.
\]
Then 
$Tq(v_e+_Ew_e)=\alpha_x+\beta_x=Tq(u_{e'}+_Ey_{e'})$
and\footnote{The tangent map $T+\colon T(E\times_M E)\to TE$ is a vector bundle homomorphism over $+\colon E\times_ME\to E$.}
\begin{equation*}
\begin{split}
&T_{(e,e')}+(v_{e}+_Ew_{e}, u_{e'}+_Ey_{e'})=T_{(e,e')}+\left((v_e, u_{e'})+_{E\times_M E}(w_e,y_{e'})\right)\\
&=\,\left(T_{(e,e')}+(v_e,u_{e'})\right)+_E\left(T_{(e,e')}+(w_e,y_{e'})\right)
\end{split}
\end{equation*}
which is written for simplicity, using the notation $T+=:+_{TM}$, 
\[(v_{e}+_Ew_{e})+_{TM}(u_{e'}+_Ey_{e'})=(v_e+_{TM}u_{e'})+_E(w_e+_{TM}y_{e'}),
\]
$+_E$ being the addition in the fibers of $TE\to E$, $+_{E\times_ME}$ the addition in the fibers of $TE\times_{TM}TE\to E\times_ME$ and $+_{TM}=T+$ being the addition in the fibers of $TE\to TM$.
Hence the two additions of the two vector bundle structures on $TE$ ``commute'', and $TE$ is a \textbf{double vector bundle} \cite{Pradines77,Mackenzie05}. All structure maps of one of the vector bundle structure are vector bundle morphisms with respect to the other one.

\bigskip

Consider the \emph{vertical space} $T^qE=\ker(Tq\colon TE\to TM) \subseteq TE$. By definition, it is a vector subbundle of $TE\to E$ of rank $k$. The intersection 
\[ T^qE\cap p_E\inv(0^E)=(Tq)\inv (0^{TM})\cap p_E\inv(0^E)
\]
is the space of vectors $v_{0^E_x}\in T_{0^E_x}E$ for $x\in M$ such that 
$T_{0^E_x}q(v_{0^E_x})=0_x\in T_xM$. Hence it is the restriction of $T^qE$ to the zero section of $E$. This intersection is the \emph{core} of the double vector bundle $E$ and it is seen as a vector bundle over $M$  \cite{Pradines77}.
In general, the core of a double vector bundle always inherits a vector bundle structure over the \textbf{double base} $M$ \cite{Pradines77}.
Here, the core of $TE$ can be identified as follows with the vector bundle $E\to M$.
The map \[\bar{}\,\colon E\to p_E^{-1}(0^E)\cap
(Tq)^{-1}(0^{TM})\]
 sending \[e\in E_x \quad \text{ to } \quad \bar
e=\left.\frac{d}{dt}\right\an{t=0}te\in T_{0^E_x}E\]
is a bijection and a linear map with respect to both additions on $TE$
since for $e,e'$ in the same fiber of $E$, 
\[ \left.\frac{d}{dt}\right\an{t=0}te+_E\left.\frac{d}{dt}\right\an{t=0}te'=\left.\frac{d}{dt}\right\an{t=0}t(e+e')=\left.\frac{d}{dt}\right\an{t=0}te+_{TM}\left.\frac{d}{dt}\right\an{t=0}te'.
\] In fact, the whole vector bundle $T^qE$ is isomorphic as a vector bundle over $E$ to the pullback $q^!E\to E$, via the morphism sending 
\[ (e,e') \in q^!E \quad \text{ to } \quad \left.\frac{d}{dt}\right\an{t=0} e+te',
\]
This isomorphism is easily seen in local coordinates, or in a local trivialisation of $E$. 
As a consequence of the isomorphism above, each section $e$ of $E\to M$ defines a ``pullback'' section $ q^!e=:e^\uparrow$ of $T^qE\to E$ via 
\[ e^\uparrow (e')=\left.\frac{d}{dt}\right\an{t=0} e'+t e_{q(e')},
\]
and since $T^qE\simeq q^!E$, these core sections span $\Gamma_E(T^qE)$. The (global) flow of such a \textbf{core vector field} is simply 
\[ \R\times E\to E, \qquad (t, e')\mapsto e'+te(q(e')).
\]
The flows of two such core vector fields commute, and so two core vector fields  always have a vanishing Lie bracket.

\subsection{Linear vector fields and linear derivations}

A vector field $\chi\colon E\to TE$ is called \textbf{linear} if it is a morphism 
\[\begin{tikzcd}
	{E} & {TE} \\
	{M} & {TM}
	\arrow["{q}"', from=1-1, to=2-1]
	\arrow["\chi", from=1-1, to=1-2]
	\arrow["{Tq}", from=1-2, to=2-2]
	\arrow["X", from=2-1, to=2-2]
\end{tikzcd}\]
of vector bundles over a vector field $X\in\mx(M)$. 
In particular, a linear vector field $\chi\in\mx^l(E)$ over the zero vector field on $E$ has by definition values in the vertical space $T^qE$. Hence there is a smooth map $F\colon E\to E$ such that $q\circ F=q$ and  
\[ \chi(e)=\left.\frac{d}{dt}\right\an{t=0} e+tF(e)
\]
for all $e\in E$.
The linearity of $\chi$ then forces $F$ to be a section of $\operatorname{End}(E)$:
\[
\left.\frac{d}{dt}\right\an{t=0} e_1+e_2+tF(e_1+e_2)=\chi(e_1+e_2)=\chi(e_1)+_{TM}\chi(e_2)=\left.\frac{d}{dt}\right\an{t=0} e_1+e_2+tF(e_1)+tF(e_2)
\]
shows that $F(e_1+e_2)=F(e_1)+F(e_2)$ for $e_1,e_2$ in the same fiber of $E$, etc. Hence linear vector fields with values in $T^qE$, also called \textbf{core-linear vector fields}, are equivalent to sections of $\operatorname{End}(E)$, via the correspondence sending $\psi\in\Gamma(\operatorname{End}(E))$ to $\tilde \psi\in \mx^l(E)$, 
\[ \tilde\psi(e)=\left.\frac{d}{dt}\right\an{t=0} e+t\psi(e)
\]
for all $e\in E$.

\bigskip

Let $E\to M$ be a smooth vector bundle. 
A \textbf{linear derivation $D$ of $E$ with symbol $X\in\mx(M)$} is an $\R$-linear map
\[ D\colon \Gamma(E)\to \Gamma(E)
\]
with 
\[D(fe)=\ldr{X}(f)e+fD(e)
\]
for all  $f\in C^\infty(M)$ and $e\in\Gamma(E)$.

The space $\lie D(E)$ of derivations of $E$ is a $C^\infty(M)$-module and a real Lie algebra via the commutator: given two derivations  $D_1$ and $D_2$ of $E$ with symbols $X_1$ and $X_2$, respectively, 
the commutator
\[ [D_1, D_2]=D_1\circ D_2- D_2\circ D_1
\]
is again a derivation, with symbol $[X_1,X_2]$. Linear derivations of $E$ with symbol zero are exactly the sections of $\operatorname{End}(E)$.
This yields a short exact sequence of $C^\infty(M)$-modules 
\begin{equation}\label{ses}0\rightarrow \Gamma(\operatorname{End}(E))\rightarrow \lie D(E)\rightarrow \mx(M)\rightarrow 0.
\end{equation}
The map $\lie D(E)\rightarrow \mx(M)$ sends a derivation to its symbol. It is surjective since  a vector bundle always admits a linear connection and a linear connection ``lifts'' vector fields on $M$ to derivations of $E$.
The latter implies in fact that $\operatorname{Der}(E)$ is (non-canonically) isomorphic to $\Gamma(\operatorname{End}(E))\oplus \mx(M)$, and so the space of sections of a vector bundle over $M$. In the following, this vector bundle is written $\operatorname{Der}(E)$ and its space of sections, i.e.~the $C^\infty(M)$-module of derivations of $E$, is  $\Gamma(\operatorname{Der}(E))=\lie D(E)$.

Given a linear derivation $D$ of $E$ with symbol $X\in\mx(M)$, the \emph{dual derivation} $D^*\colon\Gamma(E^*)\to\Gamma(E^*)$ is defined by 
\[ \langle D^*\epsilon, e\rangle=X\langle \epsilon, e\rangle -\langle \epsilon, De\rangle
\]
for all $\epsilon\in\Gamma(E^*)$ and $e\in\Gamma(E)$. The symbol of $D^*$ remains the vector field $X\in\mx(M)$. 

\bigskip

\begin{theorem}[\cite{Mackenzie05}]
Let $q\colon E\to M$ be a smooth vector bundle.
A
linear vector field $\chi\in\mx^l(E)$ over a vector field $X\in\mx(M)$ defines a derivation $D_\chi\colon \Gamma(E) \to \Gamma(E)$ with symbol $X\in
\mx(M)$, via 
\begin{equation}\label{ableitungen}
\chi(\ell_{\varepsilon}) 
= \ell_{D_\chi^*(\varepsilon)} \,\,\,\, \text{ and }  \,\,\, \chi(q^*f)= q^*(X(f))
\end{equation}
for all $\varepsilon\in\Gamma(E^*)$ and $f\in C^\infty(M)$.
Conversely, given a derivation $D$ of $E$ with symbol $X$, then \eqref{ableitungen}
defines a linear vector field $\widehat D$.

This yields an isomorphism 
\[ \mx^l(E)\to \lie D(E), \qquad \chi\mapsto D_\chi
\]
of $C^\infty(M)$-modules, with inverse $D\mapsto \widehat D$.
\end{theorem}

\begin{proof}
Since $\chi$ is linear over $X$, it is in particular $q$-related to $X$; i.e.~$Tq\circ \chi=X\circ q$.
Hence the equality $ \chi(q^*f)= q^*(X(f))$ holds for all $f\in C^\infty(M)$. The additivity of $\chi$ can be written 
\[ T\!\!+\circ(\chi, \chi)=\chi\circ +\colon E\times_M E\to TE
\]
for all $x\in M$. That is, $(\chi,\chi)\colon E\times_ME\to TE\times_{TM}TE=T(E\times_ME)$ is $+$-related to $\chi$.
This implies that $\chi(\ell_\epsilon)$ is linear for all $\epsilon\in\Gamma(E^*)$: The additivity follows from $\ell_\epsilon\circ +=+\circ (\ell_\epsilon, \ell_\epsilon)\colon E\times_M E\to \mathbb R$, by the linearity of $\ell_\epsilon$:
\begin{equation*}
\begin{split}
(\chi(\ell_\epsilon))\circ +=+^*(\chi(\ell_\epsilon))=(\chi, \chi)(+^*\ell_\epsilon)=(\chi,\chi)(+\circ (\ell_\epsilon, \ell_\epsilon))=+\circ (\chi(\ell_\epsilon), \chi(\ell_\epsilon)).
\end{split}
\end{equation*}
The $\mathbb R$-homogeneity follows as usual with an algebraic induction, or similarly with the $\mathbb R$-homogeneity of $\chi$.

Write then  $D^*_\chi\epsilon$ for the section of $E^*$ defined by $\chi(\ell_\epsilon)=\ell_{D^*_\chi\epsilon}$.
First check that \eqref{ableitungen} indeed defines a derivation $D_\chi^*$ of $E^*$ with symbol $X$.
Choose $f\in C^\infty(M)$ and $\epsilon\in\Gamma(E^*)$.
Then 
\[ \chi(\ell_{f\cdot\epsilon})=\chi(q^*f\cdot \ell_\epsilon)=\chi(q^* f)\cdot \ell_\epsilon+q^*f\cdot \chi(\ell_\epsilon)=q^*(X(f))\cdot \ell_\epsilon+q^*f\cdot \ell_{D^*_\chi\epsilon}=
\ell_{X(f)\epsilon+fD^*_\chi\epsilon}.
\]
Hence $D^*_\chi(f\epsilon)=X(f)\epsilon+fD^*_\chi\epsilon$, which shows that $D^*_\chi$ is a derivation of $E^*$.

Conversely, since a vector field on $E$ is completely determined by its values on $C^\infty_{\rm lin}(E)$ and on $q^*C^\infty(M)$ because $E$ has a smooth atlas with pullback and linear functions only, and since the two conditions in \eqref{ableitungen} are easily seen to be compatible, a derivation $D$ of $E$ with symbol $X\in\mx(M)$ defines a linear vector field $\widehat D$. The isomorphism of $\mx^l(E)$ with $\lie D(E)$ as $C^\infty(M)$-modules is then immediate.
\end{proof}

Given a derivation $D$ over $X\in\mx(M)$, the explicit formula for $\widehat D\in\mx^l(E)$ is 
\begin{equation}\label{explicit_hat_D}
\widehat D(e_m)=T_me(X(m))+_E\left.\frac{d}{dt}\right\an{t=0}(e_m-tD(e)(m))
\end{equation}
for $e_m\in E$ and any $e\in\Gamma(E)$ such that $e(m)=e_m$. This can be checked by explicit computations on linear and pullback functions on $E$:
on $q^*f$ for $f\in C^\infty(M)$
\begin{equation*}
\begin{split}
\left(T_me(X(m))+_E\left.\frac{d}{dt}\right\an{t=0}(e_m-tD(e)(m))\right)(q^*f)=X(m)(f)=\left(\widehat D(e_m)\right)(q^*f),
\end{split}
\end{equation*}
and on $\ell_\epsilon$ for $\epsilon \in\Gamma(E^*)$
\begin{equation*}
\begin{split}
&\left(T_me(X(m))+_E\left.\frac{d}{dt}\right\an{t=0}(e_m-tD(e)(m))\right)(\ell_\epsilon)\\
&=T_m\langle\epsilon,e\rangle(X(m))+\left.\frac{d}{dt}\right\an{t=0}(\langle \epsilon(m), e_m\rangle -t\langle \epsilon(m), D(e)(m)\rangle)\\
&=X(m)\langle\epsilon,e\rangle-\langle \epsilon(m), D(e)(m)\rangle=\langle (D^*\epsilon)(m), e_m\rangle=\ell_{D^*\epsilon}(e_m)=\left(\widehat D(e_m)\right)(\ell_\epsilon).
\end{split}
\end{equation*}
Since $E$ has a smooth atlas with linear and pullback coordinate functions, this shows \eqref{explicit_hat_D}.

As a consequence, the following proposition can be proved, see \cite{Mackenzie05}.
\begin{proposition}
Let $q\colon E\to M$ be a smooth vector bundle and choose derivations $D$, $D_1$ and $D_2$ of $E$ with symbols $X$, $X_1$ and $X_2$, respectively,
as well as sections $e, e_1, e_2\in\Gamma(E)$.
Then 
\begin{equation}\label{brackets}
\begin{split}
\left[\widehat{D_1}, \widehat{D_2}\right]&=\widehat{[D_1,D_2]},\\
\left[\widehat{D}, e^\uparrow\right]=(De)^\uparrow &\text{ and } \left[e_1^\uparrow,e_2^\uparrow\right]=0.
\end{split}
\end{equation}
\end{proposition}
The first equality shows that the map $\lie D(E)\to \mx^l(E)$, $D\mapsto\widehat{D}$ is an isomorphism of \emph{Lie algebras}.

\begin{proof}
The third equation was already proved. For the second one, compute for $f\in C^\infty(M)$ and $\epsilon\in\Gamma(E^*)$:
\[\left[\widehat{D}, e^\uparrow\right](q^* f)=\widehat{D}(0)-e^\uparrow(q^*(X(f))=0=(De)^\uparrow(q^*f)
\]
and 
\begin{equation*}
\begin{split}\left[\widehat{D}, e^\uparrow\right](\ell_\epsilon)&=\widehat{D}(q^*<\epsilon, e>)-e^\uparrow(\ell_{D^*\epsilon})=q^*(X<\epsilon, e>-<D^*\epsilon, e>)\\
&=q^*<\epsilon, De>=(De)^\uparrow(\ell_\epsilon).
\end{split}
\end{equation*}
Similarly compute for the first equation
\[\left[\widehat{D_1}, \widehat{D_2}\right](q^* f)=\widehat{D_1}(q^*(X_2(f)))-\widehat{D_2}(q^*(X_1(f))=q^*([X_1,X_2](f))=\widehat{\left[D_1, D_2\right]}(q^*f)
\]
and 
\begin{equation*}
\begin{split}
\left[\widehat{D_1}, \widehat{D_2}\right](\ell_\epsilon)&=\widehat{D_1}(\ell_{D_2^*\epsilon})-\widehat{D_2}(\ell_{D_1 ^*\epsilon})=\ell_{(D_1^*D_2^*-D_2^*D_1^*)\epsilon}
\\
&=\ell_{[D_1,D_2]^*\epsilon}=\widehat{[D_1,D_2]}(\ell_\epsilon).\qedhere
\end{split}
\end{equation*}
\end{proof}

\bigskip

\section{On the flow of a linear vector field}\label{Section_linear_flows}

This section proves the main theorem of this note.
\begin{theorem}\label{lemma_linear_flow}
Let $E\to M$ be a smooth vector bundle and let $\chi\in\mx^l(E)$ be a linear vector field on $E$ over $X\in\mx(M)$. Let $\phi\colon \Omega\to M$ be the flow of $X$, with, as usual, $\Omega\subseteq \R\times M$ an open subset containing $\{0\}\times M$. Then the flow $\Phi$ of $\chi\in\mx^l(E)$ is defined on 
\[ \widetilde\Omega:=\left\{ (t, e)\in \R\times E \mid (t,q(e))\in \Omega
\right\}
\]
and for each $t\in\R$, the map
\[ \Phi_t\colon q\inv(M_t)\to q\inv(M_{-t})
\]
is an isomorphism of vector bundles over the diffeomorphism $\phi_t\colon M_t\to M_{-t}$, where 
$M_t\subseteq M$ is the open subset
\[ M_t:=\{x\in M\mid (t,x)\in\Omega\}.
\]
\end{theorem}

\begin{proof}
First, as usual $Tq\circ \chi=X\circ q$ implies that each flow line of $\chi$ projects under $q$ to a flow line of $X$, and so that 
\[ \widetilde \Omega\subseteq \{(t, e)\in \mathbb R\times E\mid (t, q(e))\in \Omega\}.
\]

Choose a point $x\in M$ and consider the maximal integral curve $c\colon J\to M$ of $X$ through $x$ at time $0$, i.e.~with an open interval $J\subseteq \mathbb R$ containing $0$. Define $C\colon J\to E$, $C(t)=0^E(c(t))$ for all $t\in J$. Then  for all $t\in J$
\[ \dot C(t)=T_{c(t)}0^E(\dot c(t))=T_{c(t)}0^E(X(c(t)))=\chi(0^E(c(t)))=\chi(C(t))
\]
since $\chi$ is linear and so in particular 
\[\begin{tikzcd}
	E & TE \\
	M & TM
	\arrow["\chi", from=1-1, to=1-2]
	\arrow["{0^E}", from=2-1, to=1-1]
	\arrow["X"', from=2-1, to=2-2]
	\arrow["{T0^E}"', shift right, from=2-2, to=1-2]
\end{tikzcd}\]
commutes. (The map $T0^E$ is the zero section of the vector bundle $Tq\colon TE\to TM$.) That is, $C\colon J\to E$ is an integral curve of $\chi$ starting at $0^E_x$ at time $0$.
This shows that if $(t,x)\in \Omega$ for $t\in\mathbb R$ and $x\in M$, then also $(t,0^E_x)\in\widetilde \Omega$, the flow domain of $\chi$.

Since $\widetilde \Omega$ is open in $\mathbb R\times E$, there exist then $I\subseteq \mathbb R$ open around $t$ and $U\subseteq E$ open around $0^E_x$ such that 
\[ (t, 0^E_x)\in I\times U\subseteq \widetilde \Omega.
\]
Since $U\subseteq E$ is open and $E$ is locally trivial, $U$ contains the elements of a basis $(e^1, \ldots, e^k)$ of $E_x$.
Then $(t,e^i)\in \widetilde \Omega$ for $i=1,\ldots, k$. Write 
\[ c^i\colon J_i\to E
\]
for the flow line of $\chi$ through $e^i$ at time $0$, hence with $0, t\in J_i$ for each $i=1,\ldots, k$. Then as already observed, the curves $q\circ c^i\colon J_i\to M$ are all integral curves of $X$ through $x$ at time $0$. They must hence coincide on the intersection
\[ I=\bigcap_{i=1}^k J_i,
\]
which is an open interval containing $0$ and $t$.
Take now an arbitrary $e\in E_x$ and write it in the basis $(e^1, \ldots, e^k)$:
\[ e=\sum_{i=1}^k \alpha_ie^i
\]
with $\alpha_1,\ldots, \alpha_k\in\mathbb R$.  
Consider the smooth curve 
\[ C\colon I\to E, \qquad s\mapsto \sum_{i=1}^k\alpha_i c^i(s).
\]
This is well-defined since the curves $c^i$ all have the same projection to $M$ on $I$.
Then $C(0)=e$ and 
\[ \dot C(s)=\alpha_1\cdot_{TM}\dot c^1(s)+_{TM}\ldots+_{TM}\alpha_k\cdot_{TM} \dot c^k(s)
\]
by definition of $\cdot_{TM}$ and $+_{TM}$.
This yields 
\[ \dot C(s)=\alpha_1\cdot_{TM}\chi(c^1(s))+_{TM}\ldots+_{TM}\alpha_k\cdot_{TM} \chi(c^k(s))=\chi\left( \sum_{i=1}^k\alpha_i c^i(s)
\right)=\chi(C(s))
\]
for all $s\in I$,
since $\chi$ is linear. This shows that $C\colon I\to E$ is an integral curve of $\chi$ through $e$ at time $0$ and which is defined at time $t$. 
Hence $(t,e)\in \widetilde \Omega$. The above shows that 
\[ \{(t, e)\in \mathbb R\times E\mid (t, q(e))\in \Omega\}\subseteq \widetilde \Omega.
\]
Hence the two sets are equal.

\medskip
For each $t\in \R$ and each $x\in M_t$ the flow $\Phi_t$ restricts to a map
\[ E\an{x}\to E\an{\phi_t(x)}
\]
which is bijective since it has as inverse the restriction of $\Phi_{-t}$ to $E\an{\phi_t(x)}$.
To see that $\Phi_t$ is linear, note that 
\[ T\!\!+\circ(\chi,\chi)\colon E\times_ME\to TE
\]
equals 
\[ \chi\circ +\colon E\times_ME\to TE.
\]
Hence the vector field $(\chi, \chi)$ on $E\times_M E$, which is an embedded submanifold of $E\times E$, is $+$-related to the vector field $\chi$ on $E$. Since the flow $(\chi, \chi)$ at a time $t$ is $(\Phi_t,\Phi_t)$, this integrates to the equality
\[ +\circ (\Phi_t,\Phi_t)=\Phi_t\circ +,
\]
which exactly restricts to the additivity of $\Phi_t\an{E\an x}$. Similarly, for each $\alpha\in \R$, the compatibility of $\chi$ with the scalar multiplication by $\alpha$ integrates to the $\mathbb R$-homogeneity of $\Phi_t\an{E\an x}$, as already seen in the considerations above.
\end{proof}

\bigskip

Theorem  \ref{lemma_linear_flow} has the following corollary, which is standard, see e.g.~\cite[IV,\textsection 1]{Lang99} or \cite[Theorem 4.31]{Lee18}.
In fact, both statements are equivalent, since Theorem  \ref{lemma_linear_flow} can also be deduced from Corollary \ref{lemma_linear_nODE}, but with more work than in the proof above.
\begin{corollary}\label{lemma_linear_nODE}
 Let $n\geq 1$.
Consider a smooth function $A\colon J\to \mathbb R^{n\times n}$ on an open interval $J=(\alpha, \beta)\subseteq \mathbb R$. Then for each $t_0\in J$ and each $v_0\in \mathbb R^n$ the nonautonomous linear ordinary differential equation 
\[ \dot x(t)=A(t)\cdot x(t), \qquad x(t_0)=v_0
\]
has a unique smooth solution $x\colon J\to \mathbb R^n$.
\end{corollary}
\begin{proof}
Consider the smooth vector field $X\colon J\times \mathbb R^n\to T(J\times \mathbb R^n)$ defined by 
\[ X(s,v)=(s,1,v,A(s)(v))\in T_sJ\times T_v\mathbb R^n=\{s\}\times \R\times\{v\}\times \R^n
\]
for all $(s,v)\in J\times \R^n$. The vector field $X$ is linear over the vector field $\partial_s\colon J\to TJ=J\times \mathbb R$, $s\mapsto (s,1)$  on $J$. 
\[\begin{tikzcd}
	{J\times\R^n} & {T(J\times\R^n)} \\
	J & TJ
	\arrow["X", curve={height=-12pt}, from=1-1, to=1-2]
	\arrow["q"', from=1-1, to=2-1]
	\arrow["p", from=1-2, to=1-1]
	\arrow["Tq", from=1-2, to=2-2]
	\arrow["{\partial_t}"', curve={height=12pt}, from=2-1, to=2-2]
	\arrow["{p_J}"', from=2-2, to=2-1]
\end{tikzcd}\]
This is easy to see: for $v_1,v_2\in V$ and $\alpha_1,\alpha_2\in\mathbb R$
\begin{equation*}
\begin{split}
X(s, \alpha_1v_1+\alpha_2v_2)&=(s,1,\alpha_1v_1+\alpha_2v_2,A(s)(\alpha_1v_1+\alpha_2v_2))\\
&=(s,1,\alpha_1v_1+\alpha_2v_2,\alpha_1A(s)(v_1)+\alpha_2A(s)(v_2))\\
&=\alpha_1\cdot_{TJ}(s,1,v_1, A(s)(v_1))+_{TJ}\alpha_2\cdot_{TJ}(s,1,v_2,A(s)(v_2))\\
&=\alpha_1\cdot_{TJ}X(s,v_1)+_{TJ}\alpha_2\cdot_{TJ}X(s,v_2).
\end{split}
\end{equation*}

Since $\partial_s\in\mx(J)$ has the flow domain 
\[ \Omega=\{(t,s)\in\mathbb R\times J\mid t+s\in J\}=\bigcup_{s\in J} \left((\alpha-s, \beta-s)\times \{s\}\right)\subseteq \mathbb R\times J
\]
and the flow 
\[ \phi\colon \Omega\to J, \qquad (t,s)\mapsto t+s,
\]
by Theorem \ref{lemma_linear_flow}, the vector field $X$ has a maximal  flow
\[ \Phi\colon \Omega\times\R^n\to J\times \R^n,  
\]
such that for each $t\in \R$,
\[\begin{tikzcd}
	{(J\cap(\alpha-t, \beta-t))\times\mathbb R^n} && {J\times \R^n} \\
	{J\cap(\alpha-t,\beta-t)} && J
	\arrow["{\Phi_t}", from=1-1, to=1-3]
	\arrow["q", from=1-1, to=2-1]
	\arrow["q", from=1-3, to=2-3]
	\arrow["{\phi_t}"', from=2-1, to=2-3]
\end{tikzcd}\]
is an isomorphism of vector bundles.

For each pair $(t_0,v_0)\in J\times \R^n$, the smooth curve\footnote{The construction of $\gamma^{(t_0,v_0)}$ is the standard construction of time-dependent flow lines of a time dependent vector field, see e.g. the proof of \cite[Theorem 9.48]{Lee13}. 
}
\[\gamma^{(t_0,v_0)}\colon J\to\R^n, \qquad s\mapsto (\pr_{\R^n}\circ\, \Phi)(s-t_0, t_0, v_0)
\]
satisfies 
\[ \gamma^{(t_0,v_0)}(t_0)=\pr_{\R^n}\left(\Phi(0, t_0, v_0)\right)=\pr_{\R^n}(t_0,v_0)=v_0.
\]
In addition, 
\[  \Phi(s-t_0, t_0, v_0)=\left(s, \gamma^{(t_0,v_0)}(s)\right)
\]
by definition of $\gamma^{(t_0,v_0)}$ and since $\Phi$ is linear over $\phi$. Hence 
for all $s\in J$
\begin{equation*}
\begin{split}
\frac{d}{ds}\gamma^{(t_0,v_0)}(s)&=\frac{d}{ds}(\pr_{\R^n}\circ\, \Phi)(s-t_0, t_0, v_0)=\pr_{T\R^n}\left(
\frac{d}{ds}\Phi(s-t_0, t_0, v_0)
\right)\\
&=\pr_{T\R^n}\left(
X\left(\Phi(s-t_0, t_0, v_0)\right)\right)=\pr_{T\R^n}\left(s, 1, \gamma^{(t_0,v_0)}(s), A(s)\left(\gamma^{(t_0,v_0)}(s)\right)
\right)\\
&=\left(\gamma^{(t_0,v_0)}(s), A(s)\left(\gamma^{(t_0,v_0)}(s)\right)
\right).\qedhere
\end{split}
\end{equation*}
\end{proof}
Note that if $A$ does not depend on $t$, then $\gamma^{(t_0,v_0)}$ is, as usual, explicitly given by 
\[ t\mapsto 
\exp((t-t_0)A)\cdot v_0,\] 
but in general the differential equation is obtained by computing the time-ordered exponential of the matrix-valued function $A$, see e.g.~\cite{Lam98}.

\bigskip

Theorem \ref{lemma_linear_flow} has as well the following immediate corollary, which will be useful later on.
\begin{corollary}\label{lemma_flat_sections}
In the situation of Theorem \ref{lemma_linear_flow}, set $\chi=\widehat{D}$ for a linear derivation $D\colon \Gamma(E)\to\Gamma(E)$ with symbol $X\in\mx(M)$. For $t\in \R$ and $e\in\Gamma(E)$ define 
\[ \Phi_t^\star e \in \Gamma_{M_t}(E)
\]
to be the smooth section given by
\[ (\Phi_t^\star e)_x= \Phi_{-t}(e(\phi_t(x)))
\]
for all $x\in M_t$. Then\footnote{This derivative is taken pointwise, i.e.~in fibers of $E$.}
\[ De=\left.\frac{d}{dt}\right\an{t=0} \Phi_t^\star e.
\]
In particular $\Phi_t^\star e = e$ on $M_t$ for all $t$ if and only if $De=0$. 
\end{corollary}

Note that Corollary \ref{lemma_flat_sections} implies as well that 
\begin{equation}\label{eq_d_flow}
\Phi_t^\star(De)=\frac{d}{dt}\Phi_t^\star e=D(\Phi_t^\star e)
\end{equation}
on $M_t$ for all $t\in\mathbb R$.

\begin{proof}
For  $x\in M$ and $e'\in E_x$ compute
\begin{equation*}
\begin{split}
(De)^\uparrow \an{e'} &=\left.\left[ \widehat D, e^\uparrow\right]\right\an{e'}=\left.\frac{d}{dt}\right\an{t=0}(\Phi_t^*(e^\uparrow))\an{e'}\\
&= \left.\frac{d}{dt}\right\an{t=0} T_{\Phi_t(e')}\Phi_{-t}\left( e^\uparrow (\Phi_t(e'))\right)\\
&=  \left.\frac{d}{dt}\right\an{t=0} \left.\frac{d}{ds}\right\an{s=0}\Phi_{-t}\left( \Phi_t(e')+se(\phi_t(x))\right)\\
&=  \left.\frac{d}{dt}\right\an{t=0} \left.\frac{d}{ds}\right\an{s=0}e'+s \Phi_{-t}(e(\phi_t(x))).
\end{split}
\end{equation*}
Take $f\in C^\infty(M)$. Then 
\begin{equation*}
\begin{split}
\left((De)^\uparrow \an{e'}\right)(q^*f) &= \left.\frac{d}{dt}\right\an{t=0} \left.\frac{d}{ds}\right\an{s=0}(q^*f)\left(e'+s \Phi_{-t}(e(\phi_t(x)))\right)\\
&=\left.\frac{d}{dt}\right\an{t=0}\left.\frac{d}{ds}\right\an{s=0}f(x)=0=\left(\left.\left(\left.\frac{d}{dt}\right\an{t=0}\Phi_t^\star e\right)^\uparrow\right\an{e'}\right)(q^*f).
\end{split}
\end{equation*}
Take $\epsilon\in\Gamma(E^*)$. Then 
\begin{equation*}
\begin{split}
\left((De)^\uparrow \an{e'}\right)(\ell_\epsilon) &= \left.\frac{d}{dt}\right\an{t=0} \left.\frac{d}{ds}\right\an{s=0}\ell_\epsilon\left(e'+s \Phi_{-t}(e(\phi_t(x)))\right)\\
&= \left.\frac{d}{ds}\right\an{s=0} \left.\frac{d}{dt}\right\an{t=0}\left(\left\langle\epsilon(x), e'\rangle +s \langle\epsilon(x),\Phi_{-t}(e(\phi_t(x)))\right\rangle\right)\\
&=\left.\frac{d}{ds}\right\an{s=0}s \left\langle\epsilon(x), \left.\frac{d}{dt}\right\an{t=0}\Phi_{-t}(e(\phi_t(x)))\right\rangle\\
&=\left\langle\epsilon(x), \left.\frac{d}{dt}\right\an{t=0}\Phi_{-t}(e(\phi_t(x)))\right\rangle
=\left(\left.\left(\left.\frac{d}{dt}\right\an{t=0}\Phi_t^\star e\right)^\uparrow\right\an{e'}\right)(\ell_\epsilon). \qedhere
\end{split}
\end{equation*}
\end{proof}

\section{Applications}
This section collects a few further  useful applications of Theorem \ref{lemma_linear_flow}.
\subsection{Vector bundles over contractible manifolds are trivial}
Theorem \ref{lemma_linear_flow} implies  the following corollary, which is standard, but the proof of which is not easy to find in the smooth case in the literature. 
\begin{corollary}\label{cor_standard_vb_over_int}
Let $M$ be a smooth manifold and let $I\subseteq \mathbb R$ be an open interval. Choose $t_0\in I$ and set $\iota\colon M\to I\times M$ to be the smooth inclusion $x\mapsto (t_0,x)$. Then 
each vector bundle $F\to I\times M$ is isomorphic to $I\times \iota^!F\to I\times M$, via an isomorphism over the identity on $I\times M$.
\end{corollary}
This implies immediately that \textbf{a smooth vector bundle over a contractible manifold is trivial}: consider a contractible smooth manifold $M$ and a vector bundle $E\to M$. Take a smooth map $\varphi\colon I\times M\to M$, with $I\subseteq \mathbb R$ an open interval containing $[0,1]$ and such that $\varphi_0\equiv x_0$ for some $x_0\in M$ and $\varphi_1=\id_M$. Let $\iota_j\colon M\to I\times M$ be the smooth inclusion $x\mapsto (j,x)$ for $j=0,1$. Then $\varphi^! E$ is a smooth vector bundle over $I\times M$. Hence by Corollary \ref{cor_standard_vb_over_int} $\varphi^!E\simeq I\times \iota_0^!(\varphi^!E)=I\times \varphi_0^!E=I\times M\times  E_{x_0}$ over the identity on $I\times M$. As a consequence, $E\simeq \id_M^!E=\varphi_1^!E=\iota_1^!\varphi^!E=\iota_1^!(I\times M\times  E_{x_0})\simeq M\times E_{x_0}$ over the identity on $M$.

\begin{proof}[Proof of Corollary \ref{cor_standard_vb_over_int}]
Consider any linear vector field $X\in\mx^l(F)$ over the vector field $\partial_t\in\mx(I\times M)$. Such a linear vector field exists since any vector field on $I\times M$ can be lifted to a linear vector field on $F$, see \eqref{ses}.
The flow of $\partial_t$ is given by 
\[ \phi\colon\Omega=\{(t,s,x)\in \mathbb R\times I\times M\mid t+s\in I\}\to I\times M, \quad (t,s,x)\mapsto (t+s,x),
\]
and so by Theorem \ref{lemma_linear_flow} the flow of $X$ is a map
\[ \Phi\colon \widetilde \Omega=\{(t,e_{(s,x)})\in \mathbb R\times F\mid t+s\in I\}\to F
\]
such that for all $t\in \mathbb R$, 
\[\begin{tikzcd}
	{F\arrowvert_{((I-t)\cap I)\times M}} && {F\arrowvert_{(I\cap (I+t))\times M}} \\
	{((I-t)\cap I)\times M} && {(I\cap (I+t))\times M}
	\arrow["{\Phi_t}", from=1-1, to=1-3]
	\arrow[from=1-1, to=2-1]
	\arrow[from=1-3, to=2-3]
	\arrow["{\phi_t}"', from=2-1, to=2-3]
\end{tikzcd}\]
is an isomorphism of vector bundles.
 The map 
\[\begin{tikzcd}
	F && {I\times \iota^!F} \\
	{I\times M} && {I\times M}
	\arrow["\Theta", from=1-1, to=1-3]
	\arrow[from=1-1, to=2-1]
	\arrow[from=1-3, to=2-3]
	\arrow["{\id_{I\times M}}"', from=2-1, to=2-3]
\end{tikzcd}\]
defined 
by 
\[ \Theta(e_{(t,x)})=\left(t, \Phi(t_0-t, e_{(t,x)})\right)
\]
is well-defined, smooth by definition, and covers the identity on $I\times M$. It is a diffeomorphism since it has the smooth inverse
\[ \Theta^{-1}\colon I\times \iota^!F\to F, \qquad \Theta^{-1}(t, e_{(t_0,x)})=\Phi(t-t_0, e_{(t_0,x)}).
\]
It remains to show that $\Theta$ is fibrewise linear. But this follows immediately from the fact that $\Phi_t$ is a vector bundle isomorphism for all $t\in\mathbb R$.
\end{proof}

\subsection{Flows of derivations on tensor products of vector bundles}\label{constr_conn}
Let $E\to M$ and $F\to M$ be two vector bundles over a same base manifold $M$.
Two linear derivations $D_ E\colon\Gamma(E)\to\Gamma(E)$ and $D_ F\colon\Gamma(F)\to\Gamma(F)$ with a same symbol $X\in\mx(M)$ define together a linear derivation $D$ on $E\otimes F$ via 
\[ D(e\otimes f)=(D_Ee)\otimes f+ e\otimes(D_Ff)
\]
for all  $e\in\Gamma(E)$ and $f\in\Gamma(F)$. This can be extended to tensor products of arbitrary many vector bundles, for instance to powers of $E^{\otimes k}\otimes (E^*)^{\otimes l}$ for $k,l\in\N$.

This section computes the flow $\Phi^{E^*}$ of $\widehat{D_ E^*}\in\mx^l(E^*)$, where $D_ E^*$ is the dual linear derivation to $D_E$  on $E^*$, as well as the flow $\Phi$ of $\widehat{D}\in\mx^l(E\otimes F)$.
Let $\phi\colon \Omega\to M$ be the flow of $X\in\mx(M)$ and let $\Phi^E\colon \widetilde\Omega^E\to E$ be the flow of $\widehat{D_ E}\in\mx^l(E)$. Let $\Phi^F\colon \widetilde\Omega^F\to F$ be the flow of $\widehat{D_F}\in\mx^l(F)$, and let $q_ *\colon E^*\to M$ be the vector bundle projection of $E^*$.

\medskip
First choose $x\in M$ and $\epsilon\in E^*\an{x}$. Consider the curve 
\[ \gamma\colon I_x\to E^*, \qquad t\mapsto \Phi_{-t}^*(\epsilon)=\epsilon\circ\Phi_{-t}\in E^*\an{\phi_t(x)},
\]
with $I_x:=\Omega\cap(\mathbb R\times\{x\})$, the domain of definition of the maximal integral curve of $\phi$ through $x$ at time $0$.
Then for $e\in\Gamma(E)$ (for simplicity a global section) and any $t\in I_x$
\begin{equation*}
\begin{split}
\dot\gamma(t)(\ell_e)&=\frac{d}{dt}\langle \gamma(t), e\an{q_*(\gamma(t))}\rangle =\frac{d}{dt}\langle\epsilon,  \Phi_{-t}(e\an{\phi_t(x)})\rangle\\
&= \frac{d}{dt}\langle \epsilon, (\Phi_t^\star e)\an{x}\rangle=\langle \epsilon, (\Phi_t^\star(D_Ee))\an{x}\rangle\\
&=\langle \gamma(t), D_Ee\an{\phi_t(x)}\rangle=\ell_{D_Ee}(\gamma(t))=\widehat{D_E^*}\an{\gamma(t)}(\ell_e)
\end{split}
\end{equation*}
by Corollary \ref{lemma_flat_sections}.
This shows that $\gamma$ is an integral curve of $\widehat{D_E^*}$ through $\epsilon$ at time $0$. Since the flow domain of $\widehat{D_E^*}$ is 
\[ \widetilde{\Omega}^{E^*}=\left\{(t,\epsilon)\in \mathbb R\times E^*\mid (t, q_*(\epsilon))\in\Omega\right\},
\]
the curve $\gamma$ must be the maximal integral curve of $\widehat{D_E^*}$ through $\epsilon$ at time $0$.
The flow $\Phi^{E^*}$ of $\widehat{D_E^*}\in\mx^l(E^*)$ is hence defined by 
\begin{equation}\label{dual_flow} \Phi^{E^*}\colon \{(t, \epsilon)\in\R\times E^*\mid (t,q_{E^*}(\epsilon))\in\Omega\}\to E^*, 
\qquad  (t,\epsilon)\mapsto \epsilon\circ \Phi^E_{-t}.
\end{equation}
Given $\epsilon\in\Gamma(E^*)$, the induced curve of sections of $E^*$ through $\epsilon$ at time $0$  is then 
\[
(\Phi^{E^*})_t^\star\epsilon\in\Gamma_{M_t}(E^*)
\]
for all $t\in\mathbb R$, 
\[ ((\Phi^{E^*})_t^\star\epsilon)_x=\Phi^{E^*}_{-t}(\epsilon(\phi_t(x)))=\epsilon\an{\phi_t(x)}\circ \Phi^E_t\in E^*\an x
\]
for all $x\in M_t$.

\medskip
Next, since the flow of $\widehat{D}\in\mx^l(E\otimes F)$ is linear, it is enough to determine it on a basis of $(E\otimes F)_x$ at each $x\in M$.
Hence it is enough to determine it on elementary tensors.
Choose $x\in M$, $e\in\Gamma(E)$ and $f\in\Gamma(F)$ global sections. Consider the curve
\[ \gamma\colon I_x\to E\otimes F,\qquad  t\mapsto \Phi^E_t(e\an{x})\otimes \Phi^F_t(f\an{x}) \in (E\otimes F)\an{\phi_t(x)}.
\]
Choose further a (global) section $\epsilon$ of $E^*$ and a (global) section $\kappa$ of $F^*$ and compute
\begin{equation*}
\begin{split}
\dot \gamma(t)(\ell_{\epsilon\otimes \kappa})&=\frac{d}{dt}\left\langle\epsilon\an{\phi_t(x)}, \,\Phi^E_t(e\an{x})\right\rangle
\cdot \left\langle\kappa\an{\phi_t(x)}, \,\Phi^F_t(f\an{x})\right\rangle\\
&=\frac{d}{dt}\left\langle((\Phi^{E^*})^\star_{t}\epsilon)\an{x}, e\an{x}\right\rangle
\cdot \left\langle((\Phi^{F^*})^\star_{t}\kappa)\an{x}, f\an{x}\right\rangle\\
&\overset{\eqref{eq_d_flow}}{=}
\left\langle ((\Phi^{E^*})^\star_{t}(D_E^*\epsilon))\an{x}, e\an{x}\right\rangle
\cdot \left\langle((\Phi^{F^*})^\star_{t}\kappa)\an{x}, f\an{x}\right\rangle\\
&\quad +\left\langle\epsilon\an{\phi_t(x)}, \,\Phi^E_t(e\an{x})\right\rangle
\cdot \left\langle((\Phi^{F^*})^\star_{t}(D_F^*\kappa))\an{x}, f\an{x}\right\rangle\\
&=
\left\langle (D_E^*\epsilon)\an{\phi_t(x)},  \,\Phi^E_t(e\an{x})\right\rangle
\cdot  \left\langle\kappa\an{\phi_t(x)}, \,\Phi^F_t(f\an{x})\right\rangle\\
&\quad +\left\langle\epsilon\an{\phi_t(x)}, \,\Phi^E_t(e\an{x})\right\rangle
\cdot \left\langle(D_F^*\kappa)\an{\phi_t(x)}, \,\Phi^F_t(f\an{x})\right\rangle\\
&=\ell_{(D_E^*\epsilon)\otimes \kappa+\epsilon\otimes(D_F^*\kappa)}(\gamma(t))
=\ell_{D^*(\epsilon\otimes \kappa)}(\gamma(t))=\widehat{D}\an{\gamma(t)}(\ell_{\epsilon\otimes\kappa})
\end{split}
\end{equation*}
for all $t\in I_x$.
Also  $\dot \gamma(t)(q_{E\otimes F}^*g)=X\an{\phi_t(x)}(g)=\widehat{D}\an{\gamma(t)}(q_{E\otimes F}^*g)$
for $g\in C^\infty(M)$.
Hence 
\[\dot\gamma(t)=\widehat{D}\an{\gamma(t)}
\]
for all $t\in I_x$ and 
the flow of $\widehat{D}\in\mx^l(E\otimes F)$ is pointwise the tensor product of the flows of $\widehat{D_E}$ and $\widehat{D_F}$:
\begin{equation}\label{tensor_flow}
\begin{split}
 \Phi\colon  \{(t, \chi)\in\R\times E\otimes F\mid (t,q_{E\otimes F}(\chi))\in\Omega\}&\to E\otimes F, \\
 (t,e\otimes f)&\mapsto (\Phi_t^Ee)\otimes(\Phi_t^Ff)
\end{split}
\end{equation}
for all elementary tensors $e\otimes f\in E\otimes F$.
Given $e\in\Gamma(E)$ and $f\in\Gamma(F)$ the 
induced curve of sections of $E\otimes F$ is then 
\[
\Phi_t^\star(e\otimes f)\in\Gamma_{M_t}(E\otimes F)
\]
for all $t\in\mathbb R$, 
\[ (\Phi_t^\star(e\otimes f))_x=(\Phi_t^\star e)\an{x}\otimes (\Phi_t^\star f)\an{x}
\]
for all $x\in M_t$.

\subsection{The kernel of a transitive Lie algebroid}
This section now provides an elementary proof of the fact that the kernel of a transitive Lie algebroid (see \cite{Mackenzie05} and references therein) is a Lie algebra bundle\footnote{Note that a bundle of Lie algebras is a Lie algebroid with trivial anchor. Its fibers are then all (isotropy) Lie algebras, which may depend on their base point. A Lie algebra bundle is a Lie algebroid with trivial anchor and isomorphic isotropy Lie algebras.}.

\begin{proposition}\label{prop_K}
Let $A\to M$ be a transitive Lie algebroid.
\begin{enumerate}
\item  The kernel 
\[ K=\ker(\rho\colon A\to TM)
\]
is a subalgebroid of $A$. Since the anchor vanishes on $K$, it is a bundle of Lie algebras.
\item $\nabla\colon\Gamma(A)\times\Gamma(K)\to \Gamma(K)$ 
defined by 
\[ \nabla_ak=[a,k]
\]
for all $a\in\Gamma(A)$ and all $k\in\Gamma(K)$ is a flat $A$-connection on $K$.
\item For all $k_1,k_2\in\Gamma(K)$ and $a\in\Gamma(A)$
\[
\nabla_a[k_1,k_2]=[\nabla_ak_1, k_2]+[k_1,\nabla_ak_2].
\]
\end{enumerate}
\end{proposition}
\begin{proof}
The proof is a simple computation which is left to the reader.
\end{proof}

The following theorem shows that the kernel $K$ of a transitive Lie algebroid $A\to M$ is not only a bundle of Lie algebras, but a Lie algebra bundle. This fact is standard, and proofs of it can be found e.g.~in \cite{Mackenzie05} and \cite{Meinrenken21}.
\begin{theorem}\label{kernel_lab}
Let $A\to M$ be a transitive Lie algebroid with anchor $\rho$ and assume that $M$ is connected. Then the kernel $K$ of $\rho$ is a Lie algebra bundle.
\end{theorem}

\begin{proof}
Since the Lie bracket on $K$ is defined pointwise, it defines a tensor $B\in\Gamma(K^*\otimes K^*\otimes K)$,
\[ B(k_1,k_2)=[k_1,k_2]
\]
for all $k_1, k_2\in\Gamma(K)$. The vector bundle $K^*\otimes K^*\otimes K$ is equipped with the $A$-connection inherited from the one on $K$ that is found in (2) of Proposition \ref{prop_K}:
for $\psi\in \Gamma(K^*\otimes K^*\otimes K)$ and $a\in\Gamma(A)$ the tensor $\nabla_a\psi$ is defined by 
\begin{equation*}
\begin{split}
(\nabla_a\psi)(k_1,k_2)&=\nabla_a(\psi(k_1,k_2))-\psi(\nabla_ak_1,k_2)-\psi(k_1,\nabla_ak_2)\\
&=[a, \psi(k_1,k_2)]-\psi[[a,k_1], k_2]-\psi[k_1,[a, k_2]]
\end{split}
\end{equation*}
for all $k_1, k_2\in\Gamma(K)$.
By (3) in Proposition \ref{prop_K}, $\nabla_aB=0$ for all $a\in\Gamma(A)$, i.e.~$B$ is flat. But by Corollary \ref{lemma_flat_sections}, $\nabla_aB=0$ just means that  the flow $\Phi$ of $\widehat{\nabla_a}\in\mx^l(K^*\otimes K^*\otimes K)$ does 
\[ \Phi_t^\star B=B
\]
on $M_t$ for all $t\in \R$. 
By the considerations in Section \ref{constr_conn}, for $t\in\R$ the section $\Phi_t^\star B$ of $K^*\otimes K^*\otimes K$ over $M_t$ is defined by
\[ (\Phi_t^\star B)\an{x}(k_1,k_2)=(\Phi_{-t}(B\an{\phi_t(x)}))(k_1,k_2)=\Psi_{-t}(B\an{\phi_t(x)}(\Psi_tk_1,\Psi_tk_2))
\]
for $x\in M_t$ and $k_1,k_2\in K\an{x}$, where the map $\Psi$ on the right-hand side is the flow of $\widehat{\nabla_a}\in\mx^l(K)$. Since this is $B\an{x}(k_1,k_2)$, 
\[ [\Psi_t k_1,\Psi_tk_2]=B\an{\phi_t(x)}(\Psi_t k_1, \Psi_tk_2)=\Psi_t(B\an{x}(k_1,k_2))=\Psi_t[k_1,k_2]
\] 
for all $x\in M_t$ and all $k_1,k_2\in K\an{x}$.

\bigskip

Let $x\in M$ and choose a smooth chart $\varphi\colon U^\varphi\to \varphi(U^\varphi)$ of $M$ centered at $x$ and such that $\varphi(U^\varphi)\subseteq \R^m$ is a cube centered at $0=\varphi(x)$. Then the smooth vector fields $\partial_i^\varphi\in \mx(U^\varphi)$ all commute for $i=1,\ldots, m$, where $m$ is the dimension of $M$. Since $\rho$ is transitive there exist sections $a_1,\ldots, a_m$ of $A$ over $U^\varphi$, such that $\rho\circ a_i=\partial_i^\varphi$ for $i=1,\ldots,m$. Consider the linear vector fields
\[ \widehat{\nabla_{a_i}}\in\mx^l\left(q_K\inv U^\varphi\right), \qquad i=1,\ldots,m.
\]
Each $y\in U^\varphi$ can be written 
\[ y=\varphi\inv(y_1e_1+\ldots+y_me_m)
\]
where $y_i=\varphi_i(y)$ for $i=1,\ldots,m$ and $(e_1,\ldots, e_m)$ is the canonical basis of $\R^m$.
That is, 
\[y=\phi^{\partial_{m}^\varphi}_{y_m}\circ \ldots \circ \phi^{\partial_{1}^\varphi}_{y_1}(x),
\]
where for $i=1,\ldots,m$ and $t\in\R$ where defined, the map $\phi^{\partial_{i}^\varphi}_{t}\colon U^\varphi\to U^\varphi$ is the flow 
\[ y\mapsto \varphi\inv(\varphi(y)+te_i)
\]
of the vector field $\partial_i^\varphi$.
The fiber of $K$ over $y$ is then isomorphic to $K_x$ via the composition of flows
\[ \Phi^{\widehat{\nabla_{a_m}}}_{y_m}\circ \ldots \circ \Phi^{\widehat{\nabla_{a_1}}}_{y_1}\colon K_x\to K_y.
\]
The map
\[ \Theta\colon q_K\inv(U^\varphi)\to U^\varphi\times K_x, \qquad k_y\mapsto \left(y,\phi^{\partial_{1}^\varphi}_{-y_1}\circ \ldots \circ \phi^{\partial_{m}^\varphi}_{-y_m}(k_y)\right)
\]
for $y\in U^\varphi$ with coordinates $\varphi^1(y)=y_1, \ldots, \varphi^m(y)=y_m$, is easily seen to be a diffeomorphism 
that restricts to an isomorphism of Lie algebras in each fiber. This completes the proof of the fact that $K$ is a Lie algebra bundle.
\end{proof}

\def\cprime{$'$} \def\polhk#1{\setbox0=\hbox{#1}{\ooalign{\hidewidth
  \lower1.5ex\hbox{`}\hidewidth\crcr\unhbox0}}} \def\cprime{$'$}
  \def\cprime{$'$}

\end{document}